\documentclass[11pt]{amsart}
\usepackage[margin=30mm]{geometry}
\usepackage{amsmath,amssymb}
\usepackage{amsthm}
\usepackage{mathrsfs}

\newtheorem{thm}{Theorem}[section]

\newtheorem{lem}{Lemma}[section]
\newtheorem{cor}{Corollary}[section]

\newtheorem{rem}{Remark}[section]

\usepackage{enumitem}

\begin{document}

\title{Shadowing, chain transitive sets with nonempty interior and attractor boundaries}
\author{Noriaki Kawaguchi}
\address{Department of Mathematical and Computing Science, School of Computing, Institute of Science Tokyo, 2-12-1 Ookayama, Meguro-ku, Tokyo 152-8552, Japan}
\email{gknoriaki@gmail.com}

\begin{abstract}
We examine certain phenomena in $C^1$-dynamics from a viewpoint of shadowing and improve a known result on hyperbolic sets. We also review a result on the stability of attractor boundaries from the same viewpoint and derive several additional results.
\end{abstract}

\maketitle

\markboth{NORIAKI KAWAGUCHI}{Shadowing, chain transitive sets with nonempty interior and attractor boundaries}

\section{Introduction}

{\em Shadowing} is an important concept in the topological theory of dynamical systems (see \cite{AH,Pi2} for background). It was introduced in the study of hyperbolic differentiable dynamics \cite{An,B} and generally refers to the situation in which coarse orbits, or {\em pseudo-orbits}, can be closely approximated by true orbits. The shadowing lemma states that a $C^1$-diffeomorphism of a closed differentiable manifold exhibits shadowing near its hyperbolic sets and plays a crucial role in the study of differentiable dynamics. Noteworthy is that the shadowing is a generic property of homeomorphisms of a closed differentiable manifold \cite{PP} and so plays also a significant role in the study of $C^0$-generic dynamics.

Let $M$ be a closed differentiable manifold. In \cite{ABD}, it is conjectured that  for a $C^1$-generic diffeomorphism $f\colon M\to M$, if the non-wandering set $\Omega(f)$ has nonempty interior, then $f$ is transitive (see also \cite{ACCO,F,Po}). It is shown in \cite{ABD} that
\begin{itemize}
\item[(1)] for a $C^1$-diffeomorphism $f\colon M\to M$, if a hyperbolic set $\Lambda$ for $f$ satisfies $\Lambda\subset\Omega(f)$ and has nonempty interior, then $\Lambda=M$ and thus $f$ is Anosov,
\item[(2)] for a generic $C^1$-diffeomorphism $f\colon M\to M$, if $f$ is tame, then any homoclinic class $H$ with nonempty interior satisfies $H=M$ and thus, in such a case, $f$ is transitive.
\end{itemize}
A result in \cite[Chapter 3]{Pi1}, seemingly unrelated to the above results, is that for a $C^0$-generic homeomorphism $f\colon M\to M$, the boundary of each attractor for $f$ is Lyapunov stable (see also \cite{H}). In this paper, we give a consistent account of these phenomena from a viewpoint of shadowing.

Given a homeomorphism of a compact metric space, we show that any chain transitive set with nonempty interior and shadowing on its neighborhood coincides with an initial and terminal chain component. If an additional shadowing condition (so-called L-shadowing) is satisfied, then the chain transitive set is clopen in the phase space. As a consequence of the shadowing lemma, we improve the above result (1). We also show that for any attractor with shadowing on its neighborhood, its boundary is chain stable, and derive several corollaries.
 
Throughout, $X$ denotes a compact metric space endowed with a metric $d$. We begin by defining shadowing and L-shadowing on a subset. Let $f\colon X\to X$ be a continuous map. For $\delta>0$, a finite sequence $(x_i)_{i=0}^k$ of points in $X$, where $k\ge1$, is called a {\em $\delta$-chain} of $f$ if
\[
\sup_{0\le i\le k-1}d(f(x_i),x_{i+1})\le\delta.
\]
For a subset $S$ of $X$, we say that $f$ has {\em shadowing} on $S$ if for any $\epsilon>0$, there is $\delta>0$ such that every $\delta$-chain $(x_i)_{i=0}^k$ of $X$
with $\{x_i\colon 0\le i\le k\}\subset S$ satisfies
\[
\sup_{0\le i\le k}d(x_i,f^i(x))\le\epsilon
\]
for some $x\in X$. In such a case, $(x_i)_{i=0}^k$ is said to be {\em $\epsilon$-shadowed} by $x$. The definition of L-shadowing can be found in \cite{ACCV}. Let $f\colon X\to X$ be a homeomorphism. For $\delta>0$, a bi-infinite sequence $(x_i)_{i\in\mathbb{Z}}$ of points in $X$ is called a {\em $\delta$-limit-pseudo orbit} of $f$ if
\[
\sup_{i\in\mathbb{Z}}d(f(x_i),x_{i+1})\le\delta
\]
and
\[
\lim_{i\to-\infty}d(f(x_i),x_{i+1})=\lim_{i\to+\infty}d(f(x_i),x_{i+1})=0.
\]
For a subset $S$ of $X$, we say that $f$ has {\em L-shadowing} on $S$ if for any $\epsilon>0$, there is $\delta>0$ such that every $\delta$-limit-pseudo orbit $(x_i)_{i\in\mathbb{Z}}$ of $f$ with $\{x_i\colon i\in\mathbb{Z}\}\subset S$ satisfies
\[
\sup_{i\in\mathbb{Z}}d(x_i,f^i(x))\le\epsilon
\]
and
\[
\lim_{i\to-\infty}d(x_i,f^i(x))=\lim_{i\to+\infty}d(x_i,f^i(x))=0
\]
for some $x\in X$. In such a case, $(x_i)_{i\in\mathbb{Z}}$ is said to be {\em $\epsilon$-limit shadowed} by $x$.

{\em Chain recurrent set} and {\em chain components} are basic objects for a global understanding of dynamical systems. The definitions are as follows. Let $f\colon X\to X$ be a continuous map. Given $x,y\in X$, the notation $x\rightarrow y$ means that for every $\delta>0$, there is a $\delta$-chain $(x_i)_{i=0}^k$ of $f$ with $x_0=x$ and $x_k=y$. We say that $f$ is {\em chain transitive} if $x\rightarrow y$ for all $x,y\in X$. The {\em chain recurrent set} $CR(f)$ for $f$ is defined by
\[
CR(f)=\{x\in X\colon x\rightarrow x\}.
\]
We define a relation $\leftrightarrow$ in
\[
CR(f)^2=CR(f)\times CR(f)
\]
by: for all $x,y\in CR(f)$, $x\leftrightarrow y$ if and only if $x\rightarrow y$ and $y\rightarrow x$. Note that $\leftrightarrow$ is a closed equivalence relation in $CR(f)^2$ and satisfies $x\leftrightarrow f(x)$ for all $x\in CR(f)$. An equivalence class $C$ of $\leftrightarrow$ is called a {\em chain component} for $f$. We denote by $\mathcal{C}(f)$ the set of chain components for $f$. The basic properties of chain components are the following
\begin{itemize}
\item $CR(f)=\bigsqcup_{C\in\mathcal{C}(f)}C$, a disjoint union,
\item $C$ is closed in $X$ and satisfies $f(C)=C$ for all $C\in\mathcal{C}(f)$,
\item $f|_C\colon C\to C$ is chain transitive for all $C\in\mathcal{C}(f)$.
\end{itemize}

Given a continuous map $f\colon X\to X$, a subset $S$ of $X$ is said to be {\em $f$-invariant} if $f(S)\subset S$. A closed $f$-invariant subset $S$ of $X$ is said to be {\em chain stable} if for any $\epsilon>0$, there is $\delta>0$ such that every $\delta$-chain $(x_i)_{i=0}^k$ of $f$ with $x_0\in S$ satisfies $d(x_k,S)\le\epsilon$. We easily that for a closed $f$-invariant subset $S$ of $X$, the following conditions are equivalent
\begin{itemize}
\item $S$ is chain stable,
\item $x\rightarrow y$ implies $y\in S$ for all $x\in S$ and $y\in X$.
\end{itemize}
Following \cite{AHK}, we say that a chain component $C\in\mathcal{C}(f)$ is {\em terminal} if $C$ is chain stable. We denote by $\mathcal{C}_{\rm ter}(f)$ the set of terminal chain  components for $f$. When $f$ is a homeomorphism, a chain component $C\in\mathcal{C}(f)$ is said to be {\em initial} if $C\in\mathcal{C}_{\rm ter}(f^{-1})$. We denote by $\mathcal{C}_{\rm int}(f)$ the set of initial chain components for $f$.

For a subset $S$ of $X$ and $r>0$, we denote by $B_r(S)$ the closed $r$-neighborhood of $S$:
\[
B_r(S)=\{x\in X\colon d(x,S)\le r\}
\]
where $d(x,S)$ is the distance of $x$ from $S$:
\[
d(x,S)=\inf_{y\in S}d(x,y).
\]
For a subset $S$ of $X$, we denote by $\overline{S}$, ${\rm int}[S]$, and $\partial S$ the closure, the interior, and the boundary of $S$, respectively. Note that
\[
\partial S=\overline{S}\setminus{\rm int}[S].
\]

The first result of this paper is the following.

\begin{thm}
Let $f\colon X\to X$ be a homeomorphism and let $\Lambda$ be a closed $f$-invariant subset of $X$. If
\begin{itemize}
\item ${\rm int}[\Lambda]\ne\emptyset$,
\item $f|_\Lambda\colon\Lambda\to\Lambda$ is chain transitive,
\item $f$ has shadowing and L-shadowing on $B_b(\Lambda)$ for some $b>0$,
\end{itemize}
then $\Lambda$ is clopen in $X$; therefore, if furthermore $X$ is connected, then $X=\Lambda$ and $f$ is a mixing homeomorphism.
\end{thm}

\begin{rem}
\normalfont
A continuous map $f\colon X\to X$ is said to be
\begin{itemize}
\item {\em mixing} if for any nonempty open subsets $U,V$ of $X$, there exists $j\ge0$ such that $f^i(U)\cap V\ne\emptyset$ for all $i\ge j$,
\item {\em chain mixing} if for any $x,y\in X$ and $\delta>0$, there exists $j\ge1$ such that for each $k\ge j$, there is a $\delta$-chain $(x_i)_{i=0}^k$ of $f$ with $x_0=x$ and $x_k=y$.
\end{itemize}
If $f$ is mixing, then $f$ is chain mixing, and the converse holds when $f$ has shadowing on $X$. By \cite[Corollary 14]{RW}, we know that whenever $X$ is connected, $f$ satisfies $X=CR(f)$ if and only if $f$ is chain transitive if and only if $f$ is chain mixing.  
\end{rem}

We state a corollary of Theorem 1.1. We say that $X$ is {\em locally connected} if for any $x\in X$ and any open subset $U$ of $X$ with $x\in U$, there is an open connected subset $V$ of $X$ such that $x\in V\subset U$.

\begin{cor}
Let $f\colon X\to X$ be a homeomorphism and let $\Lambda$ be a closed $f$-invariant subset of $X$. If
\begin{itemize}
\item $X$ is connected and locally connected,
\item ${\rm int}[\Lambda]\cap{\rm int}[CR(f)]\ne\emptyset$ (in particular, $\Lambda\subset CR(f)$ and ${\rm int}[\Lambda]\ne\emptyset$),
\item $f$ has shadowing and L-shadowing on $B_b(\Lambda)$ for some $b>0$,
\end{itemize}
then, $X=\Lambda$ and $f$ is a mixing homeomorphism.
\end{cor}

By Corollary 1.1 and Lemma 2.2 in Section 2, we obtain the following corollary which improves \cite[Theorem 1]{ABD}.

\begin{cor}
Let $M$ be a closed Riemannian manifold and let $f\colon M\to M$ be a $C^1$-diffeomorphism. For a hyperbolic set $\Lambda$ for $f$, if ${\rm int}[\Lambda]\cap{\rm int}[CR(f)]\ne\emptyset$ (in particular, $\Lambda\subset CR(f)$ and ${\rm int}[\Lambda]\ne\emptyset$), then $M=\Lambda$ and $f$ is a mixing Anosov diffeomorphism.
\end{cor}

Given a homeomorphism $f\colon X\to X$, we say that a closed subset $\Lambda$ of $X$ with $f(\Lambda)=\Lambda$ is an {\em attractor} for $f$ if there is an open subset $U$ of $X$ such that $f(\overline{U})\subset U$ and
\[
\Lambda=\bigcap_{i\ge0}f^i(U).
\]
Let $M$ be a closed differentiable manifold. By \cite[Theorem 3.2.4]{Pi1}, we know that for a $C^0$-generic homeomorphism $f\colon M\to M$, the boundary of each attractor for $f$ is Lyapunov stable (see also \cite[Theorem 2.1]{H}). Note that a $C^0$-generic homeomorphism $f\colon M\to M$ has shadowing on $M$ \cite{PP}.

The second main result of this paper is the following.

\begin{thm}
Let $f\colon X\to X$ be a homeomorphism and let $\Lambda$ be an attractor for $f$. If $f$ has shadowing on $B_b(\Lambda)$ for some $b>0$, then $\partial\Lambda$ is chain stable.
\end{thm}

The proofs of Theorem 1.1 and Theorem 1.2 will be based on a same lemma (Lemma 3.1 in Section 3).

This paper consists of six sections and an appendix. In Section 2, we prove several basic lemmas. In Section 3, we present a simple lemma which will be used in the proofs of Theorem 1.1 and Theorem 1.2. In Section 4, we prove Theorem 1.1 and Corollary 1.1; and two additional corollaries are given. In Section 5, we prove Theorem 1.2 and present related corollaries. In Section 6, we make some concluding remarks. In Appendix A, we prove a theorem concerning the connectedness of attractor boundaries.

\section{Preliminaries}

The aim of this section is to prove several basic lemmas on expansiveness, L-shadowing, and chain components. 

Given a homeomorphism $f\colon X\to X$ and a subset $S$ of $X$, we say that $f$ is {\em expansive} on $S$ if there is $e>0$ such that
\[
\sup_{i\in\mathbb{Z}}d(f^i(x),f^i(y))\le e
\]
implies $x=y$ for all $x,y\in\bigcap_{i\in\mathbb{Z}}f^i(S)$. In this case, $e>0$ is called an expansive constant for $f$ on $S$.

\begin{lem}
Let $f\colon X\to X$ be a homeomorphism and let $S$ be a subset of $X$. For any $0<b<c$, if $f$ is expansive and has shadowing on $B_c(S)$, then $f$ has L-shadowing on $B_b(S)$.
\end{lem}

\begin{proof}
Let $e>0$ be an expansive constant for $f$ on $B_c(S)$. Let $0<\epsilon\le\min\{e,c-b\}$. Since $f$ has shadowing on $B_c(S)$, there is $\delta>0$ such that every $\delta$-chain of $f$ in $B_c(S)$ is $\epsilon$-shadowed by some point of $X$. For any $\delta$-limit-pseudo orbit $\xi=(x_i)_{i\in\mathbb{Z}}$ of $f$ in $B_b(S)$, we have
\[
\sup_{i\in\mathbb{Z}}d(x_i,f^i(x))\le\epsilon
\]
for some $x\in X$. Note that $\{f^i(x)\colon i\in\mathbb{Z}\}\subset B_c(S)$. If
\[
\limsup_{i\to-\infty}d(x_i,f^i(x))>a>0,
\]
then there are a sequence $0\ge i_1>i_2>\cdots$ and $y,z\in X$ such that
\begin{itemize}
\item $\lim_{j\to\infty}x_{i_j}=y$ and $\lim_{j\to\infty}f^{i_j}(x)=z$,
\item $d(y,z)\ge a$; therefore, $y\ne z$.
\end{itemize}
As $\xi$ is a $\delta$-limit pseudo orbit of $f$, we have
\[
\lim_{j\to\infty}d(f^i(x_{i_j}),x_{i+i_j})=0
\]
for all $i\in\mathbb{Z}$. Since
\[
f^i(y)=\lim_{j\to\infty}f^i(x_{i_j})=\lim_{j\to\infty}x_{i+i_j},
\]
$i\in\mathbb{Z}$, we have $\{f^i(y)\colon i\in\mathbb{Z}\}\subset B_b(S)$. Since
\[
f^i(z)=\lim_{j\to\infty}f^i(f^{i_j}(x))=\lim_{j\to\infty}f^{i+{i_j}}(x),
\]
$i\in\mathbb{Z}$, we have $\{f^i(z)\colon i\in\mathbb{Z}\}\subset B_c(S)$. Because
\[
d(f^i(x_{i_j}),f^i(f^{i_j}(x)))=d(f^i(x_{i_j}),x_{i+i_j})+d(x_{i+i_j},f^{i+i_j}(x))\le d(f^i(x_{i_j}),x_{i+i_j})+\epsilon,
\]
$i\in\mathbb{Z}$, letting $j\to\infty$, we obtain $d(f^i(y),f^i(z))\le\epsilon\le e$ for all $i\in\mathbb{Z}$. This implies $y=z$, a contradiction. It follows that
\[
\limsup_{i\to-\infty}d(x_i,f^i(x))=0.
\]
If
\[
\limsup_{i\to+\infty}d(x_i,f^i(x))>a>0,
\]
then there are a sequence $0\le i_1<i_2<\cdots$ and $y,z\in X$ such that
\begin{itemize}
\item $\lim_{j\to\infty}x_{i_j}=y$ and $\lim_{j\to\infty}f^{i_j}(x)=z$,
\item $d(y,z)\ge a$; therefore, $y\ne z$.
\end{itemize}
By the same argument as above, we obtain $y=z$, a contradiction. It follows that
\[
\limsup_{i\to+\infty}d(x_i,f^i(x))=0.
\]
As a consequence, $\xi$ is $\epsilon$-limit shadowed by $x$. Since $\xi$ is arbitrary, we conclude that $f$ has L-shadowing on $B_b(S)$, completing the proof.
\end{proof}

By Lemma 2.1 and the shadowing lemma for hyperbolic sets (see, e.g., \cite[Theorem 18.1.2]{KH}), we obtain the following lemma.

\begin{lem}
Let $M$ be a closed Riemannian manifold and let
\[
f\colon M\to M
\]
be a $C^1$-diffeomorphism. For any hyperbolic set $\Lambda$ for $f$, $f$ is expansive and has shadowing on $B_c(\Lambda)$ for some $c>0$, and therefore $f$ has shadowing and L-shadowing on $B_b(\Lambda)$ for all $0<b<c$.
\end{lem}

\begin{lem}
Let $f\colon X\to X$ be a continuous map and let $C\in\mathcal{C}(f)$. For any $b>0$, there is $a>0$ such that $B_a(C)\cap D\ne\emptyset$ implies $D\subset B_b(C)$ for all $D\in\mathcal{C}(f)$.
\end{lem}

\begin{proof}
Assume the contrary, then there are a sequence $0<a_1>a_2>\cdots$ such that $\lim_{j\to\infty}a_j=0$, and a sequence $D_j\in\mathcal{C}(f)$, $j\ge1$, such that $B_{a_j}(C)\cap D_j\ne\emptyset$ and $D_j\setminus B_b(C)\ne\emptyset$ for all $j\ge1$. For every $j\ge1$, we take $x_j\in B_{a_j}(C)\cap D_j$ and $y_j\in D_j\setminus B_b(C)$. We may assume that $\lim_{j\to\infty}x_j=x$ and $\lim_{j\to\infty}y_j=y$ for some $x,y\in X$. Then, we have $x\in C$ and $d(y,C)\ge b$; therefore, $y\notin C$. For any $j\ge1$, as $x_j,y_j\in D_j$, we have $x_j\leftrightarrow y_j$. Since $\leftrightarrow$ is a closed relation in $CR(f)^2$, it follows that $x\leftrightarrow y$. By $x\in C$ and $x\leftrightarrow y$, we obtain $y\in C$, a contradiction, thus the lemma has been proved.  
\end{proof}

\begin{lem}
Let $f\colon X\to X$ be a homeomorphism and let $C\in\mathcal{C}(f)$. If $f$ has L-shadowing on $B_b(C)$ for some $b>0$, then $C$ is clopen in $CR(f)$.
\end{lem}

\begin{proof}
Let $\epsilon>0$. Since $f$ has L-shadowing on $B_b(C)$, there is $\delta>0$ such that every $\delta$-limit-pseudo orbit of $f$ in $B_b(C)$ is $\epsilon$-limit shadowed by some point of $X$. By Lemma 2.3,  there is $a>0$ such that $B_a(C)\cap D\ne\emptyset$ implies $D\subset B_b(C)$ for all $D\in\mathcal{C}(f)$. Assume that $C$ is not clopen in $CR(f)$, then there are $x\in C$ and $y\in CR(f)\setminus C$ such that
\[
d(x,y)\le\min\{\delta,a\}.
\]
By taking $D\in\mathcal{C}(f)$ with $y\in D$, we obtain $C\ne D$ and $y\in B_a(C)\cap D\ne\emptyset$; therefore, $D\subset B_b(C)$. We define $\xi=(x_i)_{i\in\mathbb{Z}}$ by $x_i=f^i(y)$ for all $i\ge0$ and $x_i=f^i(x)$ for all $i<0$. Since $\xi$ is a $\delta$-limit-pseudo orbit of $f$ in $B_b(C)$, we obtain
\[
\lim_{i\to-\infty}d(f^i(x),f^i(z))=\lim_{i\to+\infty}d(f^i(y),f^i(z))=0
\]
for some $z\in X$. It follows that
\[
\lim_{i\to-\infty}d(f^i(z),C)=\lim_{i\to+\infty}d(f^i(z),D)=0;
\]
therefore, $s\rightarrow t$ for some $s\in C$ and $t\in D$. Similarly, we define $\xi'=(y_i)_{i\in\mathbb{Z}}$ by $y_i=f^i(x)$ for all $i\ge0$ and $y_i=f^i(y)$ for all $i<0$. Since $\xi'$ is a $\delta$-limit-pseudo orbit of $f$ in $B_b(C)$, we obtain
\[
\lim_{i\to-\infty}d(f^i(y),f^i(w))=\lim_{i\to+\infty}d(f^i(x),f^i(w))=0
\]
for some $w\in X$. It follows that
\[
\lim_{i\to-\infty}d(f^i(w),D)=\lim_{i\to+\infty}d(f^i(w),C)=0;
\]
therefore, $u\rightarrow v$ for some $u\in D$ and $v\in C$. By $s\rightarrow t$ and $u\rightarrow v$, we obtain $C=D$, a contradiction. This completes the proof of the lemma.
\end{proof}

\section{A lemma}

In this section, we prove a lemma. Our proofs of Theorem 1.1 and Theorem 1.2 will be based on the following simple observation.

\begin{lem}
Let $f\colon X\to X$ be a continuous map. For $x,y\in X$ and a closed subset $S$ of $X$, if
\begin{itemize}
\item $f(S)\subset S$ (resp.\:$f^{-1}(S)\subset S$),
\item $x\in S$, $y\in X\setminus S$, and $x\rightarrow y$ (resp.\:$y\rightarrow x$),
\item $f$ has shadowing on $B_b(S)$ for some $b>0$,
\end{itemize}
then $x\in\partial S$.
\end{lem}

\begin{proof}
We first consider the case where $f(S)\subset S$ and $x\rightarrow y$. Assume the contrary, i.e., $x\notin\partial S$ and so $x\in{\rm int}[S]$, to obtain a contradiction. We take $0<a<\min\{b,d(y,S)\}$ such that
\[
\max\{d(p,S),d(f(p),q)\}\le a
\]
implies $d(q,S)\le b$ for all $p,q\in X$. Let $a\ge\delta_1>\delta_2>\cdots$ and $\lim_{j\to\infty}\delta_j=0$. Since $x\rightarrow y$, for every $j\ge1$, there is a $\delta_j$-chain $(x_i^{(j)})_{i=0}^{k_j}$ of $f$ such that $x_0^{(j)}=x$ and $x_{k_j}^{(j)}=y$. Given $j\ge1$, let
\[
l_j=\min\{i\ge1\colon d(x_i^{(j)},S)>a\}
\]
and note that $d(x_{l_j}^{(j)},S)>a$. Since $d(x_{l_j-1}^{(j)},S)\le a$ and
\[
d(f(x_{l_j-1}^{(j)}),x_{l_j}^{(j)})\le\delta_j\le a,
\]
we have $d(x_{l_j}^{(j)},S)\le b$ and so $\{x_i^{(j)}\colon0\le i\le l_j\}\subset B_b(S)$. Taking a subsequence if necessary, we may assume that $\lim_{j\to\infty}x_{l_j}^{(j)}=z$ for some $z\in X$. Then, $d(z,S)\ge a$ and for every $\delta>0$, there is a $\delta$-chain $(y_i)_{i=0}^l$ of $f$ such that $\{y_i\colon0\le i\le l\}\subset B_b(S)$, $y_0=x$, and $y_l=z$. Fix $0<\epsilon<a$ with $B_\epsilon(x)\subset S$. Since $f$ has shadowing on $B_b(S)$, there is $\delta>0$ such that every $\delta$-chain of $f$ in $B_b(S)$ is $\epsilon$-shadowed by some point of $X$. We take a $\delta$-chain $(y_i)_{i=0}^l$ of $f$ such that $\{y_i\colon0\le i\le l\}\subset B_b(S)$, $y_0=x$, and $y_l=z$. Then, 
\[
\max\{d(x,v),d(z,f^l(v))\}\le\epsilon
\]
for some $v\in X$. It follows that $v\in S$ and so $f^l(v)\in S$; thus, we obtain $d(z,S)\le\epsilon<a$, a contradiction.

We next consider the case where $f^{-1}(S)\subset S$ and $y\rightarrow x$. Assume the contrary, i.e., $x\notin\partial S$ and so $x\in{\rm int}[S]$, to obtain a contradiction. We take $0<a<\min\{b,d(y,S)\}$ such that
\[
\max\{d(p,S),d(f(q),p)\}\le a
\]
implies $d(q,S)\le b$ for all $p,q\in X$. Let $a\ge\delta_1>\delta_2>\cdots$ and $\lim_{j\to\infty}\delta_j=0$. Since $y\rightarrow x$, for every $j\ge1$, there is a $\delta_j$-chain $(x_i^{(j)})_{i=0}^{k_j}$ of $f$ such that $x_0^{(j)}=y$ and $x_{k_j}^{(j)}=x$. Given $j\ge1$, let
\[
l_j=\max\{i\le k_j-1\colon d(x_i^{(j)},S)>a\}
\]
and note that $d(x_{l_j}^{(j)},S)>a$. Since $d(x_{l_j+1}^{(j)},S)\le a$ and
\[
d(f(x_{l_j}^{(j)}),x_{l_j+1}^{(j)})\le\delta_j\le a,
\]
we have $d(x_{l_j}^{(j)},S)\le b$ and so $\{x_i^{(j)}\colon l_j\le i\le k_j\}\subset B_b(S)$. Taking a subsequence if necessary, we may assume that $\lim_{j\to\infty}x_{l_j}^{(j)}=z$ for some $z\in X$. Then, $d(z,S)\ge a$ and for every $\delta>0$, there is a $\delta$-chain $(y_i)_{i=0}^l$ of $f$ such that $\{y_i\colon0\le i\le l\}\subset B_b(S)$, $y_0=z$, and $y_l=x$. Fix $0<\epsilon<a$ with $B_\epsilon(x)\subset S$. Since $f$ has shadowing on $B_b(S)$, there is $\delta>0$ such that every $\delta$-chain of $f$ in $B_b(S)$ is $\epsilon$-shadowed by some point of $X$. We take a $\delta$-chain $(y_i)_{i=0}^l$ of $f$ such that $\{y_i\colon0\le i\le l\}\subset B_b(S)$, $y_0=z$, and $y_l=x$. Then, 
\[
\max\{d(z,v),d(x,f^l(v))\}\le\epsilon
\]
for some $v\in X$. It follows that $f^l(v)\in S$ and so $v\in S$; thus, we obtain $d(z,S)\le\epsilon<a$, a contradiction. This completes the proof of the lemma.
\end{proof}

\section{Proof of Theorem 1.1 and corollaries}

In this section, we prove Theorem 1.1 and Corollary 1.1. We also present two corollaries. Let us first prove the following lemma via Lemma 3.1.

\begin{lem}
Let $f\colon X\to X$ be a continuous map and let $\Lambda$ be a closed $f$-invariant subset of $X$. For a closed subset $S$ of $X$, if
\begin{itemize}
\item $f|_\Lambda\colon\Lambda\to\Lambda$ is chain transitive,
\item $\Lambda\subset S$ and $f(S)\subset S$ (resp.\:$f^{-1}(S)\subset S$),
\item there are $x\in\Lambda$ and $y\in X\setminus S$ such that $x\rightarrow y$ (resp.\:$y\rightarrow x$),
\item $f$ has shadowing on $B_b(S)$ for some $b>0$,
\end{itemize}
then $\Lambda\subset\partial S$.
\end{lem}

\begin{proof}
Consider the case where $f(S)\subset S$ and $x\rightarrow y$. Since $f|_\Lambda\colon\Lambda\to\Lambda$ is chain transitive, for any $z\in\Lambda$, we have $z\rightarrow x$. By $z\rightarrow x$ and $x\rightarrow y$, we obtain $z\rightarrow y$. From Lemma 3.1, it follows that $z\in\partial S$; therefore, we obtain $\Lambda\subset\partial S$.

Consider the case where $f^{-1}(S)\subset S$ and $y\rightarrow x$. Since $f|_\Lambda\colon\Lambda\to\Lambda$ is chain transitive, for any $z\in\Lambda$, we have $x\rightarrow z$. By $y\rightarrow x$ and $x\rightarrow z$, we obtain $y\rightarrow z$. From Lemma 3.1, it follows that $z\in\partial S$; therefore, we obtain $\Lambda\subset\partial S$, proving the lemma.
\end{proof}

By Lemma 4.1, we obtain the following lemma.

\begin{lem}
Let $f\colon X\to X$ be a continuous map and let $\Lambda$ be a closed $f$-invariant subset of $X$. If
\begin{itemize}
\item ${\rm int}[\Lambda]\ne\emptyset$,
\item $f|_\Lambda\colon\Lambda\to\Lambda$ is chain transitive,
\item $f$ has shadowing on $B_b(\Lambda)$ for some $b>0$,
\end{itemize}
then $\Lambda\in\mathcal{C}_{\rm ter}(f)$.
\end{lem}

\begin{proof}
Since $f|_{\Lambda}$ is chain transitive, we have $\Lambda\subset C$ for some $C\in\mathcal{C}(f)$. Fix $r\in\Lambda$ and note that $r\rightarrow s$ for all $s\in C$. If $\Lambda$ is chain stable, then it follows that $C\subset\Lambda$ and so $\Lambda=C\in\mathcal{C}_{\rm ter}(f)$. Thus, it is sufficient to show that $\Lambda$ is chain stable. Assume the contrary, i.e., $\Lambda$ is not chain stable. Then, there are $x\in\Lambda$ and $y\in X\setminus\Lambda$ such that $x\rightarrow y$. Letting $S=\Lambda$, we see that the assumptions of Lemma 4.1 is satisfied. It follows that $\Lambda\subset\partial\Lambda$ (and so $\Lambda=\partial\Lambda$), which contradicts the assumption that ${\rm int}[\Lambda]\ne\emptyset$; therefore, the lemma has been proved.
\end{proof}

Similarly, by Lemma 4.1, we obtain the following lemma.

\begin{lem}
Let $f\colon X\to X$ be a homeomorphism and let $\Lambda$ be a closed $f$-invariant subset of $X$. If
\begin{itemize}
\item ${\rm int}[\Lambda]\ne\emptyset$,
\item $f|_\Lambda\colon\Lambda\to\Lambda$ is chain transitive,
\item $f$ has shadowing on $B_b(\Lambda)$ for some $b>0$,
\end{itemize}
then $\Lambda\in\mathcal{C}_{\rm ini}(f)\cap\mathcal{C}_{\rm ter}(f)$.
\end{lem}

\begin{proof}
By Lemma 4.2, we obtain $\Lambda\in\mathcal{C}_{\rm ter}(f)$. Note that $f(\Lambda)=\Lambda$ and so $f^{-1}(\Lambda)=\Lambda$. If $\Lambda\notin\mathcal{C}_{\rm ini}(f)$, then there are $x\in\Lambda$ and $y\in X\setminus\Lambda$ such that $y\rightarrow x$. Letting $S=\Lambda$, we see that the assumptions of Lemma 4.1 is satisfied. It follows that $\Lambda\subset\partial\Lambda$ (and so $\Lambda=\partial\Lambda$), which contradicts the assumption that ${\rm int}[\Lambda]\ne\emptyset$; therefore, we obtain $\Lambda\in\mathcal{C}_{\rm ini}(f)$, completing the proof of the lemma.
\end{proof}

\begin{rem}
\normalfont
Let $M$ be a closed differentiable manifold. In \cite[Lemma 3.3]{ABD}, it is shown that for a $C^1$-generic diffeomorphism $f\colon M\to M$, any homoclinic class $H$ with nonempty interior is Lyapunov stable for $f$ and $f^{-1}$.
\end{rem}

Given a continuous map $f\colon X\to X$ and $x\in X$, let $\omega(x,f)$ denote the $\omega$-limit set of $x$ for $f$, i.e., the set of $y\in X$ such that $\lim_{j\to\infty}f^{i_j}(x)=y$ for some $0\le i_1<i_2<\cdots$. Note that $\omega(x,f)$ is a closed $f$-invariant subset of $X$ and satisfies
\[
\lim_{i\to\infty}d(f^i(x),\omega(x,f))=0.
\]
Because $y\rightarrow z$ for all $y,z\in\omega(x,f)$, there is a unique $C(x,f)\in\mathcal{C}(f)$ such that $\omega(x,f)\subset C(x,f)$.

\begin{lem}
Let $f\colon X\to X$ be a homeomorphism and let $C\in\mathcal{C}(f)$. If $C\in\mathcal{C}_{\rm ini}(f)\cap\mathcal{C}_{\rm ter}(f)$ and $C$ is clopen in $CR(f)$, then $C$ is clopen in $X$.
\end{lem}

\begin{proof}
Since $C$ is clopen in $CR(f)$, there is $r>0$ such that for any $D\in\mathcal{C}(f)$, $B_r(C)\cap D\ne\emptyset$ implies $C=D$. Since $C\in\mathcal{C}_{\rm ter}(f)$ ($C$ is chain stable), there is $a>0$ such that $\{f^i(x)\colon i\ge0\}\subset B_r(C)$ for all $x\in B_a(C)$. Given $x\in B_a(C)$, we take $C(x,f)\in\mathcal{C}(f)$ such that $\omega(x,f)\subset C(x,f)$ and so
\[
\lim_{i\to\infty}d(f^i(x),C(x,f))=0.
\]
As $\omega(x,f)\subset B_r(C)$, we have $B_r(C)\cap C(x,f)\ne\emptyset$ and so $C=C(x,f)$. It follows that
\[
\lim_{i\to\infty}d(f^i(x),C)=0;
\]
therefore, $x\rightarrow y$ for some $y\in C$. By $C\in\mathcal{C}_{\rm ini}(f)$, we obtain $x\in C$. Since $x\in B_a(C)$ is arbitrary, we conclude that $B_a(C)\subset C$, thus, $C$ is open (and so clopen) in $X$, completing the proof.
\end{proof}

Let us complete the proof of Theorem 1.1.

\begin{proof}[Proof of Theorem 1.1]
Since ${\rm int}[\Lambda]\ne\emptyset$, $f|_\Lambda\colon\Lambda\to\Lambda$ is chain transitive, and $f$ has shadowing on $B_b(\Lambda)$, by Lemma 4.3, we have $\Lambda\in\mathcal{C}_{\rm ini}(f)\cap\mathcal{C}_{\rm ter}(f)$. Since $f$ has L-shadowing on $B_b(\Lambda)$, by Lemma 2.4, $\Lambda$ is clopen in $CR(f)$. By using Lemma 4.4, we conclude that $X=\Lambda$, completing the proof of the theorem.
\end{proof}

Here, we give a proof of Corollary 1.1.

\begin{proof}[Proof of Corollary 1.1]
We fix $x\in{\rm int}[\Lambda]\cap{\rm int}[CR(f)]$. As $x\in{\rm int}[CR(f)]$ and $X$ is locally connected, there is an open connected subset $V$ of $X$ such that $x\in V\subset CR(f)$. By taking $C\in\mathcal{C}(f)$ with $x\in C$, we obtain $V\subset C$ and so ${\rm int}[C]\ne\emptyset$. Since
$f(\Lambda)\subset\Lambda$, $x\in{\rm int}[\Lambda]$, and $f$ has shadowing on $B_b(\Lambda)$, Lemma 3.1 implies that $y\in\Lambda$ for all $y\in X$ with $x\rightarrow y$. For any $y\in C$, by $x,y\in C$, we obtain $x\rightarrow y$ and so $y\in\Lambda$. It follows that $C\subset\Lambda$. Since  
\begin{itemize}
\item $X$ is connected,
\item ${\rm int}[C]\ne\emptyset$,
\item $f|_C\colon C\to C$ is chain transitive,
\item $f$ has shadowing and L-shadowing on $B_b(C)$,
\end{itemize}
by Theorem 1.1, we conclude that $X=C$ (and so $X=\Lambda$) and $f$ is a mixing homeomorphism, completing the proof. 
\end{proof}

We shall present two corollaries.

\begin{cor}
Let $f\colon X\to X$ be a homeomorphism. If
\begin{itemize}
\item $X$ is connected,
\item ${\rm int}[CR(f)]\ne\emptyset$,
\item $\mathcal{C}(f)$ is a finite set, 
\item $f$ has shadowing on $X$,
\end{itemize}
then $f$ is a mixing homeomorphism.
\end{cor}

\begin{proof}
Since ${\rm int}[CR(f)]\ne\emptyset$ and $\mathcal{C}(f)$ is a finite set, we have ${\rm int}[C]\ne\emptyset$ for some $C\in\mathcal{C}(f)$. As $f$ has shadowing on $X$, by Lemma 4.3, we obtain $C\in\mathcal{C}_{\rm ini}(f)\cap\mathcal{C}_{\rm ter}(f)$. Since $\mathcal{C}(f)$ is a finite set and so $C$ is clopen in $CR(f)$, by Lemma 4.4, $C$ is clopen in $X$. Because $X$ is connected, we conclude that $X=C$ and $f$ is mixing, proving the corollary.
\end{proof}

Given a homeomorphism $f\colon X\to X$, we easily see that if $f$ has L-shadowing on $X$, then $f$ has shadowing on $X$. If $f$ has L-shadowing on $X$, then by Lemma 2.4, every $C\in\mathcal{C}(f)$ is clopen in $CR(f)$; therefore, $\mathcal{C}(f)$ is a finite set. From Corollary 4.1, we obtain the following corollary. 

\begin{cor}
Let $f\colon X\to X$ be a homeomorphism. If
\begin{itemize}
\item $X$ is connected,
\item ${\rm int}[CR(f)]\ne\emptyset$, 
\item $f$ has L-shadowing on $X$,
\end{itemize}
then $f$ is a mixing homeomorphism.
\end{cor}

At the end of this section, we give a direct proof of Theorem 1.1 using Lemma 4.3. For a homeomorphism $f\colon X\to X$ and $x\in X$, we define subsets $W^u(x),W^s(x)$ of $X$ by
\[
W^u(x)=\{y\in X\colon\lim_{i\to-\infty}d(f^i(x),f^i(y))=0\},
\]
\[
W^s(x)=\{y\in X\colon\lim_{i\to+\infty}d(f^i(x),f^i(y))=0\}.
\]

\begin{proof}[A direct proof of Theorem 1.1]
By Lemma 4.3, we obtain $\Lambda\in\mathcal{C}_{\rm ini}(f)\cap\mathcal{C}_{\rm ter}(f)$. It follows that $W^u(x)\cup W^s(x)\subset\Lambda$ for all $x\in\Lambda$. Let $\epsilon>0$. Since $f$ has L-shadowing on $B_b(\Lambda)$, there is $\delta>0$ such that every $\delta$-limit-pseudo orbit of $f$ in $B_b(\Lambda)$ is $\epsilon$-limit shadowed by some point of $X$. Since $\Lambda\in\mathcal{C}_{\rm ter}(f)$ ($\Lambda$ is chain stable), there is $a>0$ such that $\{f^i(y)\colon i\ge0\}\subset B_b(\Lambda)$ for all $y\in B_a(\Lambda)$. Given any $x\in\Lambda$ and $y\in X$ with $d(x,y)\le\min\{\delta,a\}$, we define $\xi=(x_i)_{i\in\mathbb{Z}}$ by $x_i=f^i(y)$ for all $i\ge0$ and $x_i=f^i(x)$ for all $i<0$. Since $\xi$ is a $\delta$-limit-pseudo orbit of $f$ in $B_b(\Lambda)$, $\xi$ is $\epsilon$-limit shadowed by some $z\in X$. Then, we have $z\in W^u(x)\cap W^s(y)$. By $x\in\Lambda$ and so $W^u(x)\subset\Lambda$, we obtain $z\in\Lambda$. It follows that $y\in W^s(z)\subset\Lambda$. Since $x\in\Lambda$ and $y\in X$ with $d(x,y)\le\min\{\delta,a\}$ are arbitrary, we conclude that $\Lambda$ is open (and so clopen) in $X$, completing the proof of Theorem 1.1.
\end{proof}

\section{Proof of Theorem 1.2 and corollaries}

In this section, we prove Theorem 1.2 and present some corollaries. For the proof of Theorem 1.2, we need a lemma.

\begin{lem}
Let $f\colon X\to X$ be a homeomorphism and let $\Lambda$ be an attractor for $f$. For any $x\in\partial\Lambda$, there is $z\in X\setminus\Lambda$ such that $z\rightarrow x$. 
\end{lem}

\begin{proof}
We take $a>0$ such that
\[
\bigcap_{i\ge0}f^i(B_a(\Lambda))=\Lambda.
\]
Let $0<\delta_1>\delta_2>\cdots$ and $\lim_{j\to\infty}\delta_j=0$. Given $j\ge1$, by $x\in\partial\Lambda$, we have $d(x,z_j)\le\min\{a,\delta_j\}$ for some $z_j\in X\setminus\Lambda$. Note that $z_j\in B_a(x)\subset B_a(\Lambda)$. As $z_j\in X\setminus\Lambda$, there is $i_j\ge1$ such that $f^{-i_j}(z)\notin B_a(\Lambda)$. Taking a subsequence if necessary, we may assume that $\lim_{j\to\infty}f^{-{i_j}}(z_j)=z$ for some $z\in X$. Then, $d(z,\Lambda)\ge a>0$; therefore, $z\in X\setminus\Lambda$. Since $\lim_{j\to\infty}\delta_j=0$ and so $\lim_{j\to\infty}z_j=x$, we obtain $z\rightarrow x$, proving the lemma.
\end{proof}

Let us prove Theorem 1.2.

\begin{proof}[Proof of Theorem 1.2]
Assume that $x\in\partial\Lambda$, $y\in X$, and $x\rightarrow y$. We shall show that $y\in\partial\Lambda$. Since $\Lambda$ is an attractor for $f$, $\Lambda$ is chain stable, which implies $y\in\Lambda$. Since $x\in\partial\Lambda$, by Lemma 5.1, we have $z\rightarrow x$ for some $z\in X\setminus\Lambda$. By $z\rightarrow x$ and $x\rightarrow y$, we obtain $z\rightarrow y$. Since
\begin{itemize}
\item $f(\Lambda)=\Lambda$ and so $f^{-1}(\Lambda)=\Lambda$,
\item $y\in\Lambda$, $z\in X\setminus\Lambda$, and $z\rightarrow y$,
\item $f$ has shadowing on $B_b(\Lambda)$,
\end{itemize}
from Lemma 3.1, it follows that $y\in\partial\Lambda$; thus, the theorem has been proved.
\end{proof}

We have the following basic lemma.

\begin{lem}
Let $f\colon X\to X$ be a continuous map and let $S$ be a closed $f$-invariant subset of $X$. If
\begin{itemize}
\item $f(S)=S$,
\item $S$ is chain stable,
\end{itemize}
then for any $a>0$, there is an attractor $\Lambda$ for $f$ such that $S\subset\Lambda\subset B_a(S)$.
\end{lem}

\begin{proof}
For $\delta>0$, let $A$ be the set of $y\in X$ such that there are $x\in S$ and a $\delta$-chain $(x_i)_{i=0}^k$ of $f$ with $x_0=x$ and $x_k=y$. Then, we see that $S\subset A$ and $f(\overline{A})\subset{\rm int}[A]$. Since $S$ is chain stable, we have $A\subset B_a(S)$ for some $\delta>0$. Letting $U={\rm int}[A]$, we obtain $f(\overline{U})\subset U$ and $S\subset U\subset B_a(S)$. Letting $\Lambda=\bigcap_{i\ge0}f^i(U)$, we see that $\Lambda$ is an attractor for $f$ and satisfies $S\subset\Lambda\subset B_a(S)$, completing the proof.
\end{proof}

By Theorem 1.2 and Lemma 5.2, we obtain the following corollary.

\begin{cor}
Let $f\colon X\to X$ be a homeomorphism and let $C\in\mathcal{C}(f)$. If
\begin{itemize}
\item X is connected,
\item $C\in\mathcal{C}_{\rm ter}(f)$ and $X\ne C$,
\item $f$ has shadowing on $B_b(C)$ for some $b>0$,
\end{itemize}
then, for every $a>0$, there is $D\in\mathcal{C}_{\rm ter}(f)$ such that ${\rm int}[D]=\emptyset$ and $D\subset B_a(C)$.
\end{cor}

\begin{proof}
Since $X\ne C$, we have $X\ne B_c(C)$ for some $c>0$. Let $0<a<\min\{b,c\}$. As $C\in\mathcal{C}_{\rm ter}(f)$ ($C$ is stable), by Lemma 5.2, there is an attractor $\Lambda$ for $f$ such that $C\subset\Lambda\subset B_a(C)$. If $\partial\Lambda=\emptyset$, then $\Lambda$ is clopen in $X$. Since $X$ is connected, this implies $X=\Lambda$ and so $X=B_c(C)$, a contradiction. It follows that $\partial\Lambda\ne\emptyset$. Since $\Lambda\subset B_a(C)$ and $f$ has shadowing on $B_b(C)$, Theorem 1.2 implies that $\partial\Lambda$ is chain stable. By taking $D\in\mathcal{C}_{\rm ter}(f|_{\partial\Lambda})$, we obtain $D\in\mathcal{C}_{\rm ter}(f)$, ${\rm int}[D]=\emptyset$, and $D\subset B_a(C)$, proving the corollary.
\end{proof}

A proof of the following lemma can be found in \cite{K}.

\begin{lem}[{\cite[Lemma 2.1]{K}}]
Let $f\colon X\to X$ be a continuous map. For any $x\in X$, there are $C\in\mathcal{C}_{\rm ter}(f)$ and $y\in C$ such that $x\rightarrow y$.
\end{lem}

We present a corollary of Lemma 5.2 and Lemma 5.3.

\begin{cor}
Let $f\colon X\to X$ be a homeomorphism and let $C\in\mathcal{C}(f)$. If
\begin{itemize}
\item $C\in\mathcal{C}_{\rm ini}(f)\cap\mathcal{C}_{\rm ter}(f)$,
\item $C$ is not clopen in $X$,
\end{itemize}
then, for every $a>0$, there are $D\in\mathcal{C}_{\rm ter}(f)\setminus\{C\}$ and $E\in\mathcal{C}_{\rm ini}(f)\setminus\{C\}$ such that $D\cup E\subset B_a(C)$.
\end{cor}

\begin{proof}
Since $C\in\mathcal{C}_{\rm ter}(f)$ ($C$ is chain stable), by Lemma 5.2, there is an attractor $\Lambda$ for $f$ such that $C\subset\Lambda\subset B_a(C)$. If $C=\Lambda$, then there is $c>0$
such that
\[
\lim_{i\to\infty}d(f^i(z),C)=0
\]
for all $z\in B_c(C)$. Since $C\in\mathcal{C}_{\rm ini}(f)$, this implies $B_c(C)\subset C$ and so $C$ is clopen in $X$. It follows that $C\ne\Lambda$. Note that $C\in\mathcal{C}_{\rm ter}(f|_\Lambda)$. If $\mathcal{C}_{\rm ter}(f|_\Lambda)=\{C\}$, then by Lemma 5.3, for any $x\in\Lambda\setminus C$, we have $x\rightarrow y$ for some $y\in C$, which contradicts $C\in\mathcal{C}_{\rm ini}(f)$. It follows that $\mathcal{C}_{\rm ter}(f|_\Lambda)\ne\{C\}$. By taking $D\in\mathcal{C}_{\rm ter}(f|_\Lambda)\setminus\{C\}$, we obtain $D\in\mathcal{C}_{\rm ter}(f)\setminus\{C\}$ and $D\subset B_a(C)$. By repeating the same argument for $f^{-1}$, we also obtain $E\in\mathcal{C}_{\rm ini}(f)\setminus\{C\}$ such that $E\subset B_a(C)$, completing the proof.
\end{proof}

\begin{rem}
\normalfont
Let $M$ be a closed differentiable manifold. In \cite[Theorem 5]{ABD}, it is shown that for a $C^1$-generic diffeomorphism $f\colon M\to M$, if a homoclinic class $H$ with $H\ne M$ is Lyapunov stable for $f$ and $f^{-1}$, then there are infinitely many (pairwise disjoint) homoclinic classes which accumulate to $H$.
\end{rem}

By Lemma 4.1 and Lemma 5.2, we shall prove the following.

\begin{lem}
Let $f\colon X\to X$ be a homeomorphism and let $C\in\mathcal{C}(f)$. If
\begin{itemize}
\item $C\notin\mathcal{C}_{\rm ini}(f)$,
\item $f$ has shadowing on $X$,
\end{itemize}
then there is an attractor $\Lambda$ for $f$ such that $C\subset\partial\Lambda$.
\end{lem}

\begin{proof}
Since $C\notin\mathcal{C}_{\rm ini}(f)$, there are $x\in C$ and $y\in X\setminus C$ such that $y\rightarrow x$. Let $S$ be the set of $w\in X$ such that $z\rightarrow w$ for some $z\in C$. If $y\in S$, then $z\rightarrow y$ for some $z\in C$. Since $x,z\in C$, we have $x\rightarrow z$. By $y\rightarrow x$ and $x\rightarrow z$, we obtain $y\rightarrow z$ and so $y\in C$, a contradiction. It follows that $y\notin S$. Since $f(C)=C$, for any $v\in C$, we have $v=f(u)$ for some $u\in C$, which implies $u\rightarrow v$ and so $v\in S$. It follows that $C\subset S$. We easily see that $f(S)=S$. Since we have $q\in S$ for all $p\in S$ and $q\in X$ with $p\rightarrow q$, $S$ is chain stable. We fix $a>0$ with $d(y,S)>a$. Since $f(S)=S$ and $S$ is chain stable, by Lemma 5.2, there is an attractor $\Lambda$ for $f$ such that $S\subset\Lambda\subset B_a(S)$. Note that $C\subset\Lambda$, $x\in C$, and $y\in X\setminus\Lambda$. By Lemma 4.1, we conclude that $C\subset\partial\Lambda$, proving the lemma.
\end{proof}

Finally, we obtain the following corollary. A similar result can be found in \cite[Theorem 6.9]{AHK}.

\begin{cor}
Let $f\colon X\to X$ be a homeomorphism and let $C\in\mathcal{C}(f)$. If $f$ has shadowing on $X$, then the following conditions are equivalent
\begin{itemize}
\item[(1)] $C\notin\mathcal{C}_{\rm ini}(f)$,
\item[(2)] there is an attractor $\Lambda$ for $f$ such that $C\subset\partial\Lambda$.
\end{itemize}
\end{cor}

\begin{proof}
By Lemma 5.4, (1) implies (2). Conversely, if there is an attractor $\Lambda$ for $f$ with $C\subset\partial\Lambda$, then by Lemma 5.1, there are $x\in C$ and $z\in X\setminus\Lambda$ such that $z\rightarrow x$. This implies $C\notin\mathcal{C}_{\rm ini}(f)$; therefore, the corollary has been proved. 
\end{proof}

\section{Remarks}

In this section, we make some concluding remarks.

\begin{rem}
\normalfont
Given a continuous map $f\colon X\to X$, a closed $f$-invariant subset $M$ of $X$ is called a {\em minimal set} for $f$ if $M=\omega(x,f)$ for all $x\in M$. Let $f\colon X\to X$ be a homeomorphism and let $M$ be a minimal set for $f$. If ${\rm int}[M]\ne\emptyset$, then letting $U={\rm int}[M]$, we see that $M=\bigcup_{i\ge0}f^{-i}(U)$; therefore, $M$ is a clopen subset of $X$. If $X$ is connected, this implies $X=M$ and so $f$ is a minimal homeomorphism. 
\end{rem}

\begin{rem}
\normalfont
Let $f\colon X\to X$ be a homeomorphism and let $C\in\mathcal{C}(f)$. If
\[
C\in\mathcal{C}_{\rm ini}(f)\:\text{(resp.\:}C\in\mathcal{C}_{\rm ter}(f)),
\]
then $\liminf_{i\to+\infty}d(f^i(x),C)>0$ (resp.\:$\liminf_{i\to-\infty}d(f^i(x),C)>0$) for all $x\in X\setminus C$.
\end{rem}

\begin{rem}
\normalfont
Given a continuous map $f\colon X\to X$ and a closed $f$-invariant subset $S$ of $X$, we define a subset $W^s(S)$ of $X$ by
\[
W^s(S)=\{x\in X\colon\lim_{i\to\infty}d(f^i(x),S)=0\}.
\]
We easily see that for any $x\in X$, $x\in W^s(S)$ if and only if $\omega(x,f)\subset S$.

Let $f\colon X\to X$ be a homeomorphism and let $S$ be a closed subset of $X$ such that $f(S)=S$ (and so $f(\partial S)=\partial S$). For any $x\in W^s(S)$, we have $\omega(x,f)\subset S$. If 
\[
\omega(x,f)\cap{\rm int}[S]\ne\emptyset,
\]
then we have $f^i(x)\in S$ for some $i\ge0$, which implies $x\in S$. It follows that if $x\in W^s(S)\setminus S$, then
\[
\omega(x,f)\subset S\setminus{\rm int}[S]=\partial S;
\]
therefore, $x\in W^s(\partial S)$. In other words, we obtain $W^s(S)\setminus S\subset W^s(\partial S)$.
\end{rem}

\begin{rem}
\normalfont
Let $f\colon X\to X$ be a continuous map. If $X$ is locally connected and ${\rm int}[CR(f)]\ne\emptyset$, then there are $x\in CR(f)$ and an open connected subset $V$ of $X$ such that $x\in V\subset CR(f)$. By taking $C\in\mathcal{C}(f)$ with $x\in C$, we have $V\subset C$ and so ${\rm int}[C]\ne\emptyset$.
\end{rem}

\appendix

\section{}

The aim of this appendix is to prove the following generalization of \cite[Theorem 2.4]{H}.

\begin{thm}
Let $f\colon X\to X$ be a homeomorphism and let $\Lambda$ be an attractor for $f$. If $X$ is locally connected, then $\partial\Lambda$ has only finitely many connected components.
\end{thm}

The proof of the following lemma is left as an exercise for the reader.

\begin{lem}
Let $A$ be a connected subset of $X$. For a subset $S$ of $X$, if $A\cap S\ne\emptyset$ and $A\setminus S\ne\emptyset$, then $A\cap\partial S\ne\emptyset$.
\end{lem}

For $x\in X$ and $r>0$, we denote by $U_r(x)$ the open $r$-ball centered at $x$:
\[
U_r(x)=\{y\in X\colon d(x,y)<r\}.
\]
For a subset $S$ of $X$ and $r>0$, we denote by $U_r(S)$ the open $r$-neighborhood of $S$:
\[
U_r(S)=\{x\in X\colon d(x,S)<r\}.
\]
By \cite[Lemma 6.4.2]{AH}, we know that if $X$ is locally connected, then there are a compatible metric $d$ on $X$ and  $\epsilon_0>0$ such that $U_\epsilon(x)$ is connected for all $x\in X$ and $0<\epsilon\le\epsilon_0$.

\begin{proof}[Proof of Theorem A.1]
Since $\Lambda$ is an attractor for $f$, there is an open subset $U$ of $X$ such that $f(\overline{U})\subset U$ and $\Lambda=\bigcap_{i\ge0}f^i(U)$. Note that
\[
f(\partial U)=f(\overline{U}\setminus U)=f(\overline{U})\setminus f(U)\subset U\setminus f^2(\overline{U}).
\]
Let $\mathcal{C}$ denote the set of connected components of $U\setminus f^2(\overline{U})$. Since $X$ is locally connected and $U\setminus f^2(\overline{U})$ is open in $X$, every $C\in\mathcal{C}$ is open in $X$. Because $f(\partial U)$ is a compact subset of $X$, the set
\[
\mathcal{D}=\{C\in\mathcal{C}\colon f(\partial U)\cap C\ne\emptyset\}
\]
is finite. Define
\[
B=\bigcup_{C\in\mathcal{D}}C
\]
and note that
\begin{itemize}
\item $f(\partial U)\subset B\subset U\setminus f^2(\overline{U})$,
\item $\mathcal{D}$ coincides with the set of connected components of $B$.
\end{itemize}
Let $\epsilon_0>0$ be a constant as above. To complete the proof, it suffices to show that for any $0<\epsilon\le\epsilon_0$, there exists $l\ge0$ such that
\[
\partial\Lambda\subset U_{2\epsilon}(f^i(B))\quad\text{and}\quad f^i(B)\subset U_{2\epsilon}(\partial\Lambda)
\]
for all $i\ge l$. If this holds, then since $f^i(B)$ is homeomorphic to $B$ and so has at most $|\mathcal{D}|$ connected components for each $i\ge0$, $\partial\Lambda$ also has at most $|\mathcal{D}|$ connected components.

Given any $x\in\partial\Lambda$, we have $U
_\epsilon(x)\setminus\Lambda\ne\emptyset$. Let $y\in U
_\epsilon(x)\setminus\Lambda$. Since $U\supset f(U)\supset\cdots$
and $\Lambda=\bigcap_{i\ge0}f^i(U)$, there is $j_x\ge0$ such that $y\in U_\epsilon(x) \setminus f^{i+1}(U)$ for all $i\ge j_x$. For every $i\ge j_x$, since
\[
x\in U_\epsilon(x)\cap\Lambda\subset U_\epsilon(x)\cap f^{i+1}(U),
\]
by Lemma A.1, we obtain
\[
U_\epsilon(x)\cap f^{i+1}(\partial U)=U_\epsilon(x)\cap\partial f^{i+1}(U)\ne\emptyset;
\]
therefore, $x\in U_\epsilon(f^{i+1}(\partial U))$. Since $\partial\Lambda$ is a compact subset of $X$, there are $x_1,x_2,\dots,x_n\in\partial\Lambda$ such that $\partial\Lambda\subset\bigcup_{m=1}^n U_\epsilon(x_m)$. Letting $j=\max_{1\le m\le n}j_{x_m}$, we obtain
\[
\partial\Lambda\subset U_{2\epsilon}(f^{i+1}(\partial U))\subset U_{2\epsilon}(f^i(B))
\]
for all $i\ge j$.

Let
\[
\Lambda_\epsilon=\{y\in\Lambda\colon d(y,\partial\Lambda)\ge\epsilon\}.
\]
For $y\in\Lambda_\epsilon$, if $U_\epsilon(y)\setminus\Lambda\ne\emptyset$, then since $y\in U_\epsilon(y)\cap\Lambda$, by using Lemma A.1, we obtain $U_\epsilon(y)\cap\partial\Lambda\ne\emptyset$; therefore, $d(y,\partial\Lambda)<\epsilon$, a contradiction. It follows that $U_\epsilon(y)\subset\Lambda$ for all $y\in\Lambda_\epsilon$. Since $\overline{U}\supset f(\overline{U})\supset\cdots$
and $\Lambda=\bigcap_{i\ge0}f^i(\overline{U})$, by compactness, there is $k\ge0$ such that
\[
f^i(\overline{U})\subset U_\epsilon(\Lambda) 
\]
for all $i\ge k$. Given any $i\ge k$ and $x\in f^i(\overline{U})\setminus\Lambda$, we take $y\in\Lambda$ such that $x\in U_\epsilon(y)$. If $y\in\Lambda_\epsilon$, then we obtain $x\in U_\epsilon(y)\subset\Lambda$, a contradiction. It follows that $y\not\in\Lambda_\epsilon$ and so $d(y,\partial\Lambda)<\epsilon$, which implies $x\in U_{2\epsilon}(\partial\Lambda)$. Since $i\ge k$ and $x\in f^i(\overline{U})\setminus\Lambda$ are arbitrary, we obtain
\[
f^i(B)\subset f^i(\overline{U})\setminus\Lambda\subset U_{2\epsilon}(\partial\Lambda)
\]
for all $i\ge k$. This completes the proof of the above claim, and thus the theorem has been proved.
\end{proof}

\begin{rem}
\normalfont
Let $S$ be a subset of $X$. The above proof shows that if $S$ is a closed subset of $X$ with
\[
f(\partial U)\subset S\subset\overline{U}\setminus\Lambda,
\]
then $\lim_{i\to\infty}f^i(S)=\partial\Lambda$ with respect to the Hausdorff distance.
\end{rem}

\begin{rem}
\normalfont
Let $S$ be a closed subset of $X$. It is easy to see that if $X$ and $\partial S$ each have only finitely many chain components, then $S$ also has only finitely many chain components.
\end{rem}

\end{document}